\documentclass[11pt,a4paper]{amsart}

\usepackage[T1]{fontenc}
\usepackage[utf8]{inputenc}
\usepackage{lmodern}
\usepackage{amsmath,amsthm,amssymb,mathtools}
\usepackage{hyperref}
\usepackage[nameinlink,noabbrev]{cleveref}
\usepackage{enumitem}
\usepackage{microtype}
\usepackage{geometry}
\hypersetup{colorlinks=true,linkcolor=blue,citecolor=blue,urlcolor=blue}

\newtheorem{theorem}{Theorem}[section]
\newtheorem{corollary}[theorem]{Corollary}
\newtheorem{proposition}[theorem]{Proposition}
\newtheorem{lemma}[theorem]{Lemma}
\theoremstyle{definition}
\newtheorem{definition}[theorem]{Definition}
\newtheorem{remark}[theorem]{Remark}

\newcommand{\eqm}{\equiv_{\mathrm m}}
\newcommand{\lem}{\le_{\mathrm m}}
\newcommand{\lfo}{\le_{\mathrm{fo}}}
\newcommand{\ce}{c.e.}
\newcommand{\rng}{\operatorname{rng}}

\title[A strengthened form of D\"egtev's argument]
{A strengthened form of D\"egtev's argument and
\texorpdfstring{$D$}{D}-maximal many-one degrees}

\author{Patrizio Cintioli}
\address{Mathematics Division, School of Science and Technology,
University of Camerino, Italy}
\email{patrizio.cintioli@unicam.it}


\begin{document}

\begin{abstract}
D\"egtev proved that if a noncomputable computably enumerable set \(U\)
satisfies his property \((R)\) and the many-one degree of \(U\) contains no
simple set, then the \(1\)-degree of \(U\) is minimal. We show that his
proof yields the stronger conclusion that, for every noncomputable
\(A\lem U\),
\[
U\le_1 A.
\]
In particular, the injectivity of the reduction \(A\le_1 U\) with which
D\"egtev begins his proof is not needed. Using the
Cholak--Gerdes--Lange classification, we show that for a noncomputable \(D\)-maximal set
\(D\), the hypothesis that its many-one degree contains no simple set is
equivalent to \(D\) being of one of Types~3--10. Hence, for every such
\(D\) and every noncomputable \(A\lem D\), one has \(D\le_1 A\).
Together with Maslova's theorem for the simple-set cases, this recovers as
a corollary the existence of a least finite-one degree in every
noncomputable c.e.\ many-one degree containing a \(D\)-maximal set.
\end{abstract}

\maketitle

\section{Introduction}

Richter, Stephan, and Zhang asked whether every noncomputable many-one
degree has a least finite-one degree inside it \cite[p.~17]{RSZ}. The
question has a negative answer even among computably enumerable (\ce)
many-one degrees \cite{CintioliNoLeastFO}. A first version of the present
preprint established a positive answer for the subclass of c.e.\ many-one
degrees containing a \(D\)-maximal set, using finite-one arguments adapted
to the different cases in the Cholak--Gerdes--Lange classification.

After that version was posted, V.~Yu.~Shavrukov pointed out in
correspondence a stronger one-one statement for Types~3--10, drew attention
to D\"egtev's 1976 paper \cite{Degtev76}, and suggested that the stronger
conclusion is essentially implicit in D\"egtev's proof. D\"egtev's
Assertion~1 \cite[p.~751]{Degtev76} states that if a noncomputable c.e.\
set \(U\) satisfies property \((R)\) and its many-one degree contains no
simple set, then the \(1\)-degree of \(U\) is minimal. His proof starts
with a noncomputable set \(A\le_1 U\) and constructs a one-one reduction
of \(U\) to \(A\).

The purpose of the present revision is to verify Shavrukov's suggestion in
detail: the same construction works when the initial hypothesis
\(A\le_1 U\) is weakened to \(A\lem U\).

\begin{theorem}\label{thm:strengthened-degtev}
Let \(U\) be a noncomputable c.e.\ set satisfying property \((R)\), and
suppose that the many-one degree of \(U\) contains no simple set. If \(A\)
is noncomputable and
\[
A\lem U,
\]
then
\[
U\le_1 A.
\]
\end{theorem}

We give a complete modernized verification, including the step
\(A\eqm U\) used at the beginning of D\"egtev's proof and the verification
that the final reduction remains injective when the initial reduction is
not assumed to be one-one.

For \(D\)-maximal sets, property \((R)\) is the superset formulation of
\(D\)-maximality. Using the Cholak--Gerdes--Lange classification, we show
that the many-one degree of a \(D\)-maximal set contains a simple set if and
only if the set is of Type~1 or Type~2. Thus
\Cref{thm:strengthened-degtev} yields a uniform conclusion for all the
remaining types.

\begin{corollary}\label{cor:cgl-one-one}
Let \(D\) be a noncomputable \(D\)-maximal set of one of Types~3--10. If \(A\) is
noncomputable and \(A\lem D\), then
\[
D\le_1 A.
\]
\end{corollary}

For Types~1--2, Maslova's theorem on simple sets supplies the finite-one
conclusion. Consequently, the result proved independently in the first
version follows from Maslova's theorem, the strengthened reading of
D\"egtev's argument established here, and the CGL classification.

\begin{corollary}\label{cor:finite-one}
Every noncomputable c.e.\ many-one degree containing a \(D\)-maximal set
contains a least finite-one degree.
\end{corollary}

It would be interesting to determine for which further natural classes of
many-one degrees the existence of a least finite-one degree still holds.

\section{Preliminaries}

All sets are subsets of $\omega$. For sets $A,B\subseteq\omega$, we write
$A\lem B$ if there is a total computable function $f$ such that
\[
x\in A \iff f(x)\in B
\]
for every $x$. If $f$ can be chosen injective, we write $A\le_1 B$. We
write $A\lfo B$ if $f$ can be chosen with all fibres finite. Thus
\[
A\le_1 B \Longrightarrow A\lfo B \Longrightarrow A\lem B.
\]
For \(\rho\in\{1,\mathrm m,\mathrm{fo}\}\), write
\[
A\equiv_\rho B
\]
if
\[
A\le_\rho B
\qquad\text{and}\qquad
B\le_\rho A.
\]
The equivalence classes induced by \(\equiv_\rho\) are the
\(\rho\)-degrees.

We say that \(A\) represents the least finite-one degree inside its
many-one degree if
\[
A\lfo B
\qquad\text{for every }B\eqm A.
\]

A c.e. set $S$ is \emph{simple} if $\overline S$ is infinite and contains no
infinite c.e. subset.

We shall use the following formulation of D\"egtev's property $(R)$.

\begin{definition}\label{def:R}
A noncomputable c.e. set $U$ satisfies \emph{property $(R)$} if, for every
c.e. set $B\supseteq U$, whenever $B\setminus U$ is not c.e. there exists a
computable set $R$ such that
\[
U\subseteq R\subseteq B.
\]
\end{definition}

This is exactly the superset formulation of $D$-maximality used by
Cholak--Gerdes--Lange; see \cite[Definition~2.4 and Lemma~2.5]{CGL}.
Throughout this paper, \(D\)-maximal sets are understood to be
noncomputable.

\section{A strengthened form of D\"egtev's argument}

We first make explicit a fact used at the beginning of D\"egtev's proof.

\begin{lemma}\label{lem:minimal-m}
Let $U$ be a noncomputable c.e. set satisfying property $(R)$. If $A$ is
noncomputable and $A\lem U$, then $U\lem A$. Hence $A\eqm U$.
\end{lemma}

\begin{proof}
Let $f$ be a many-one reduction of $A$ to $U$, and put
\[
F:=\rng(f).
\]
Since $U$ is c.e., the set $A$ is c.e. We first claim that $F\setminus U$
is not c.e. Indeed, for every $x$,
\[
x\notin A
\iff f(x)\notin U
\iff f(x)\in F\setminus U.
\]
Thus
\[
\overline A=f^{-1}(F\setminus U).
\]
If $F\setminus U$ were c.e., then $\overline A$ would be c.e., and hence $A$
would be computable, a contradiction.

Apply property $(R)$ to the c.e. set $U\cup F$. There is a computable set
$C$ such that
\[
U\subseteq C\subseteq U\cup F.
\]
Fix $a\in A$ and $b\notin A$, which exist because $A$ is noncomputable.
We define a total computable function $h$. On input $x$, if $x\notin C$,
set $h(x)=b$. If $x\in C$, run in parallel an enumeration of $U$ and the
search for a $y$ such that $f(y)=x$. Since $C\subseteq U\cup F$, at least
one of these searches terminates. If $x$ enters $U$ first, set $h(x)=a$;
if a preimage $y$ is found first, set $h(x)=y$ (ties may be resolved in
favour of the first alternative).

If $x\notin C$, then $x\notin U$ and $h(x)=b\notin A$. If $x\in C$ and
$h(x)=a$, then $x\in U$ and $h(x)\in A$. Finally, if $h(x)=y$ with
$f(y)=x$, then
\[
h(x)=y\in A \iff f(y)=x\in U.
\]
Therefore
\[
x\in U \iff h(x)\in A
\]
for all $x$, so $U\lem A$.
\end{proof}

\begin{proof}[Proof of \Cref{thm:strengthened-degtev}]
Fix a total computable many-one reduction $f$ of $A$ to $U$. By
\Cref{lem:minimal-m}, $A\eqm U$. Hence $A$ is not simple, since by
hypothesis the many-one degree of $U$ contains no simple set.

The set $A$ is infinite and c.e., so choose an infinite computable set
\[
A_1\subseteq A.
\]
Since $A$ is noncomputable, $\overline A$ is infinite; since $A$ is not
simple, $\overline A$ contains an infinite c.e. subset, and therefore an
infinite computable set
\[
A_2\subseteq\overline A.
\]
Put $F=\rng(f)$. As in the proof of \Cref{lem:minimal-m}, $F\setminus U$ is
not c.e. Applying property $(R)$ to $U\cup F$, choose a computable set
$A_3$ such that
\[
U\subseteq A_3\subseteq U\cup F.
\]
The set $A_3$ is infinite because it contains the noncomputable c.e. set $U$.

We construct a total computable function $g$ recursively, preserving the
following properties after $g(0),\ldots,g(n-1)$ have been defined:
\begin{enumerate}[label=(\roman*),leftmargin=2.4em]
\item the values $g(0),\ldots,g(n-1)$ are pairwise distinct;
\item for every $x<n$,
\[
x\in U \iff g(x)\in A;
\]
\item if $x<n$ and $g(x)\notin A_1\cup A_2$, then
\[
g(x)=\min\{y:f(y)=x\}.
\]
\end{enumerate}

Suppose $g$ has been defined below $n$.

If $n\notin A_3$, then $n\notin U$. Let $g(n)$ be the least element of
$A_2$ not among $g(0),\ldots,g(n-1)$. This value exists because $A_2$ is
infinite. It is fresh, belongs to $\overline A$, and therefore preserves
(i)--(iii).

Assume now that \(n\in A_3\). Run an enumeration of \(U\) in parallel
with the sequential search, in increasing order of \(y\), for the least
\(y\) such that \(f(y)=n\). Since
\[
A_3\subseteq U\cup F,
\]
at least one of the two searches terminates (ties are resolved in favour
of the first alternative).

If $n$ enters $U$ first, let $g(n)$ be the least unused element of $A_1$.
Then $g(n)\in A$, the value is fresh, and (i)--(iii) are preserved.

Otherwise, suppose that
\[
y_0:=\min\{y:f(y)=n\}
\]
is found first. There are three cases.

If $y_0\in A_1$, then $y_0\in A$, so $n=f(y_0)\in U$. Let $g(n)$ be the
least unused element of $A_1$.

If $y_0\in A_2$, then $y_0\notin A$, so $n=f(y_0)\notin U$. Let $g(n)$ be
the least unused element of $A_2$.

Finally, suppose $y_0\notin A_1\cup A_2$, and set
\[
g(n):=y_0.
\]
Then
\[
g(n)\in A \iff f(g(n))=n\in U,
\]
so (ii) holds, and (iii) holds by definition. It remains only to check that
$y_0$ has not been used before. 
If \(y_0=g(x)\) for some \(x<n\), then
\(g(x)\notin A_1\cup A_2\), and the inductive hypothesis (iii) gives
\[
f(g(x))=x.
\]
Hence
\[
f(y_0)=x.
\]
But \(f(y_0)=n\), so \(x=n\), contradicting \(x<n\).
Thus the new value is fresh and (i) is also preserved.

This completes the construction. Each stage is effective: membership in
$A_1,A_2,A_3$ is decidable, unused elements of $A_1$ and $A_2$ can be found
effectively, and the parallel search terminates because
$A_3\subseteq U\cup F$. Hence $g$ is total computable. By (i), it is
injective, and by (ii),
\[
x\in U \iff g(x)\in A
\]
for every $x$. Therefore $U\le_1 A$.
\end{proof}

\begin{remark}[A slightly more general form]
The proof actually uses the hypothesis that the many-one degree of \(U\)
contains no simple set only to ensure that \(A\) is not simple. Thus the
same argument proves the slightly more general statement that, if \(U\)
satisfies property \((R)\), \(A\lem U\), and \(A\) is noncomputable and
nonsimple, then \(U\le_1 A\).
\end{remark}

\begin{remark}[On D\"egtev's proof]\label{rem:degtev-text}
The proof of Assertion~1 \cite[pp.~751--752]{Degtev76} starts with a noncomputable set
\(A\le_1 U\). The preceding proof shows that the injectivity of this
initial reduction is not used: a many-one reduction suffices. The English
version consulted here states that \(f(\mathbb N)\setminus U\) is r.e.;
however, both the argument and the application of property \((R)\)
immediately following it require this set to be \emph{not} r.e., as proved
above.
\end{remark}

\section{Translation into the CGL classification}

We use the standard facts that, among \(D\)-maximal sets, Type~1 is
exactly the simple case, and Type~2 is characterized by the existence of
an infinite computable set
\[
R\subseteq\overline D
\]
such that \(D\sqcup R\) is simple; see
\cite[Theorem~3.10, Definition~3.11, and Lemma~3.17]{CGL}.

For \(D\)-maximal sets we shall also use
\cite[Theorem~4.1(1),(2)]{CGL}.

\begin{proposition}\label{prop:simple-degree}
Let $D$ be a $D$-maximal set. The many-one degree of $D$ contains a simple set if
and only if $D$ is of Type~1 or Type~2. Consequently, the $D$-maximal sets
of Types~3--10 are exactly those whose many-one degree contains no simple
set.
\end{proposition}

\begin{proof}
If \(D\) is of Type~1, then \(D\) itself is simple.
Suppose now that \(D\) is of Type~2.
Choose an infinite computable \(R\subseteq\overline D\) such that
\[
S:=D\sqcup R
\]
is simple.

Since $D$ is noncomputable, choose $a\in D$, and since $S$ is
simple choose $b\notin S$. Define
\[
\alpha(x)=
\begin{cases}
x,&x\notin R,\\
b,&x\in R,
\end{cases}
\qquad
\beta(x)=
\begin{cases}
x,&x\notin R,\\
a,&x\in R.
\end{cases}
\]
Then $\alpha$ reduces $D$ to $S$ and $\beta$ reduces $S$ to $D$, so
$D\eqm S$. Thus the many-one degree of $D$ contains a simple set.

Conversely, suppose that a simple set $S$ satisfies $S\eqm D$. Fix a
many-one reduction $f$ of $S$ to $D$, and put $F=\rng(f)$. As before,
$F\setminus D$ is not c.e., since otherwise
\[
\overline S=f^{-1}(F\setminus D)
\]
would be c.e. Since $D$ is $D$-maximal, property $(R)$ gives a computable
set $C$ such that
\[
D\subseteq C\subseteq D\cup F.
\]
The set $C\setminus D$ is infinite; otherwise $D$ would be a finite variant
of the computable set $C$ and hence computable.

We claim that $D$ is simple in $C$, i.e. that $C\setminus D$ contains no
infinite c.e. subset. Suppose, towards a contradiction, that $E$ is an
infinite c.e. subset of $C\setminus D$. Since
\[
C\setminus D\subseteq F\setminus D,
\]
every element of $E$ has an $f$-preimage. Enumerate $E$ without repetition
as $e_0,e_1,\ldots$ and, for each $i$, let
\[
y_i:=\min\{y:f(y)=e_i\}.
\]
Then the $y_i$ are pairwise distinct, form an infinite c.e. set, and satisfy
$y_i\notin S$ because $f(y_i)=e_i\notin D$. This contradicts the simplicity
of $S$. Hence $D$ is simple in $C$.

It follows that \(D\) is a maximal subset of \(C\). Indeed, let
\(D\subseteq W\subseteq C\) be c.e. If \(W\setminus D\) is c.e., then
simplicity in \(C\) makes it finite. Otherwise property \((R)\) gives a
computable set \(H\) with
\[
D\subseteq H\subseteq W.
\]
Then the computable set
\[
C\setminus H\subseteq C\setminus D
\]
is finite, and hence \(C\setminus W\) is finite.

Put \(R_0:=\overline C\). Then \(R_0\) is computable. We claim that
\(D\cup R_0\) is maximal. Indeed,
\[
\overline{D\cup R_0}=C\setminus D
\]
is infinite, since \(D\) is simple in \(C\). Let \(W\) be c.e.\ with
\[
D\cup R_0\subseteq W.
\]
Then
\[
D\subseteq W\cap C\subseteq C.
\]
Since \(D\) is a maximal subset of \(C\), either
\[
(W\cap C)\setminus D
\]
is finite or
\[
C\setminus W
\]
is finite. In the first case,
\[
W\setminus(D\cup R_0)=(W\cap C)\setminus D
\]
is finite, while in the second case
\[
\overline W=C\setminus W
\]
is finite. Hence \(D\cup R_0\) is maximal.
If \(R_0\) is infinite, \cite[Theorem~4.1(2)]{CGL} implies that
\(D\) is of Type~2. If \(R_0\) is finite, then \(D\) is a finite variant of
the maximal set \(D\cup R_0\), hence is itself maximal; therefore
\cite[Theorem~4.1(1)]{CGL} implies that \(D\) is of Type~1.
This proves the converse.
\end{proof}

\begin{proof}[Proof of \Cref{cor:cgl-one-one}]
Every $D$-maximal set satisfies property $(R)$ by
\cite[Definition~2.4 and Lemma~2.5]{CGL}. By
\Cref{prop:simple-degree}, if $D$ is of one of Types~3--10, then its
many-one degree contains no simple set. The conclusion now follows from
\Cref{thm:strengthened-degtev}.
\end{proof}

In particular, the published conclusion of D\"egtev's Assertion~1 is
recovered in the following CGL formulation.

\begin{corollary}\label{cor:minimal-one}
Every $D$-maximal set of one of Types~3--10 has minimal $1$-degree.
\end{corollary}

\begin{proof}
Let $A\le_1 D$ be noncomputable. Then $A\lem D$, so
\Cref{cor:cgl-one-one} gives $D\le_1 A$. Hence $A\equiv_1 D$.
\end{proof}

\section{The finite-one consequence}

We use Maslova's theorem that every noncomputable c.e. simple set represents
the least finite-one degree inside its many-one degree \cite{Maslova79}; see
also \cite[p.~3]{RSZ}.

\begin{proof}[Proof of \Cref{cor:finite-one}]
Let $\mathbf d$ be a noncomputable c.e. many-one degree containing a
$D$-maximal set $D$.

If $D$ is of Type~1, then $D$ is simple, so Maslova's theorem applies. If
$D$ is of Type~2, the proof of \Cref{prop:simple-degree} gives a simple set
$S\eqm D$, and Maslova's theorem applied to $S$ yields a least finite-one
degree inside $\mathbf d$.

Suppose finally that $D$ is of one of Types~3--10. Let $B\eqm D$. Then $B$
is noncomputable, and in particular $B\lem D$. By
\Cref{cor:cgl-one-one},
\[
D\le_1 B,
\]
so $D\lfo B$. Since $B\eqm D$ was arbitrary, $D$ itself represents the
least finite-one degree inside $\mathbf d$.
\end{proof}

\section*{Acknowledgments}

The author is grateful to V.~Yu.~Shavrukov for pointing out the stronger
one-one statement for Types~3--10, for drawing attention to D\"egtev's
1976 paper and suggesting that the stronger conclusion is essentially
implicit in D\"egtev's proof, and for sharing an alternative argument for
the Type~3--10 case. These observations prompted a substantial revision
of the present paper.

The first version of this paper was developed with extensive use of
ChatGPT Pro (OpenAI) and Gemini Deep Think (Google DeepMind) in exploratory
analysis and technical verification. In preparing the present
substantially revised version, ChatGPT with the GPT-5.6 Sol model (OpenAI)
was used primarily for technical checking and manuscript preparation.
The author has reworked and verified the arguments presented here and bears
sole responsibility for the mathematical content of this version.

\end{document}